\documentclass[reqno]{amsart}

\usepackage{graphicx,amsmath,amssymb,amsthm}
\usepackage[normalem]{ulem} 
\usepackage{enumitem} 

\newtheorem{thm}{Theorem}
\newtheorem{lem}[thm]{Lemma}
\theoremstyle{definition}
\newtheorem{defn}[thm]{Definition}
\theoremstyle{remark}
\newtheorem{rem}[thm]{Remark}
\newtheorem{prop}[thm]{Proposition}
\newtheorem{cor}[thm]{Corollary}
\newtheorem{remark}[thm]{Remark}
\theoremstyle{remark}

\newcommand{\norm}[1]{\left\Vert#1\right\Vert}
\newcommand{\RR}{\mathbb{R}}
\newcommand{\itemizeEqnVSpacing}{\rule{0pt}{1pt}\vspace*{-12pt}}
\newcommand{\incidences}{\mathbf I}
\newcommand{\dt}{(\delta,t)}
\newcommand{\WC}{\mathcal W}
\newcommand{\BC}{\mathcal B}
\newcommand{\WB}{(\mathcal W,\mathcal B)}
\newcommand{\Gammabar}{\textrm{\sout{$\phantom{.}$}$\!\!\Gamma$}}
\newcommand{\bdry}{\operatorname{bdry}}
\newcommand{\zar}{\operatorname{Zar}}
\newcommand{\rect}{\mathbf{R}}
\begin{document}
\title{On the Wolff Circular Maximal Function}
\author[J.~Zahl]{Joshua Zahl}%
\address{Department of Mathematics, UCLA, Los Angeles CA 90095-1555, USA}
\email{jzahl@math.ucla.edu}
\date{\today}
\begin{abstract}%
We prove sharp $L^3$ bounds for a variable coefficient generalization of the Wolff circular maximal function $M^\delta f$. For each fixed radius $r$, $M^\delta f(r)$ is the maximal average of $f$ over the $\delta$--neighborhood of a circle of radius $r$ and arbitrary center. In this paper, we consider maximal averages over families of curves satisfying the cinematic curvature condition, which was first introduced by Sogge to generalize the Bourgain circular maximal function. Our proof manages to avoid a key technical lemma in Wolff's original argument, and thus our arguments also yield a shorter proof of the boundedness of the (conventional) Wolff circular maximal function. At the heart of the proof is an induction argument that employs an efficient partitioning of $\mathbb{R}^3$ into cells using the discrete polynomial ham sandwich theorem.
\end{abstract}%

\maketitle
\section{Introduction}
In \cite{Wolff1}, Wolff considered the following maximal function:
\begin{equation}\label{WolffCircMaxmlFnCircles}
M^\delta f(r)=\sup_{x\in\RR^2}\frac{1}{|C^\delta(x,r)|}\int_{C^{\delta}(x,r)}|f(y)|dy,
\end{equation}
where $C^{\delta}(x,r)$ is the $\delta$--neighborhood of a circle centered at $x$ of radius $r$. This maximal function has the same relationship to Besicovich-Rado-Kinney (BRK) sets (compact subsets of the plane containing a circle of every radius $1/2\leq r\leq 1$) as the Kakeya maximal function has to Kakeya sets. In particular, a bound of the form
\begin{equation}\label{L3BoundCircles}
\norm{M^\delta f}_{L^p([1/2,1])}\leq C_\epsilon\delta^{-\epsilon}\norm{f}_{L^p(\RR^2)}
\end{equation}
for some value of $p$ and all $\epsilon>0$ would imply that every BRK set has Hausdorff dimension 2. See \cite{Wolff2} for further details. By considering the examples where $f$ is the characteristic function of a ball of radius $\delta$ and a rectangle of dimensions $\delta\times\sqrt\delta$, we can see that $p=3$ is the smallest value of $p$ for which \eqref{L3BoundCircles} can hold. In \cite{Wolff1}, Wolff proved \eqref{L3BoundCircles} for $p=3$.

In a similar vein, Wolff and Kolasa considered the more general class of maximal functions
\begin{equation}\label{CurveMaximalFunctionDefn}
M_\Phi^\delta f(r) = \sup_{x\in U_1}\frac{1}{|\Gammabar^\delta(x,r)|}\int_{\Gammabar^{\delta}(x,r)}|f(y)|dy.
\end{equation}
Here, $U_1$ is a sufficiently small neighborhood of a point $a\in\RR^2,$ and $\Gammabar^\delta(x,r)$ is the $\delta$--neighborhood of the curve
\begin{equation}
\Gammabar(x,r) = \{y\in U_2\colon \Phi(x,y)=r\},
\end{equation}
where $U_2$ is a sufficiently small neighborhood of a point $b\in\RR^2$ and
\begin{equation*}
\Phi\colon\RR^2\times\RR^2\to\RR
\end{equation*}
is a smooth function satisfying Sogge's \emph{cinematic curvature} conditions at the point $(a,b)$:
\begin{itemize}%
\item
\itemizeEqnVSpacing%
\begin{equation}\label{cinematicCurvatureGradientCondition}
\nabla_y\Phi(a,b)\neq 0.
\end{equation}
$\phantom{.}$
\item%
\itemizeEqnVSpacing%
\begin{equation}\label{cinematicCurvatureCondition}%
 \det\bigg(\nabla_x\left[\begin{array}{c}
e\cdot \nabla_y\Phi(x,y)\\
e\cdot \nabla_y\big(\frac{e\cdot
\nabla_y\Phi(x,y)}{|\nabla_y\Phi(x,y)|}\big)
\end{array}\right]\bigg|_{(x,y)=(a,b)}\bigg)\neq 0,
\end{equation}
where $e$ is a unit vector orthogonal to $\nabla_y\Phi(a,b)$.
\end{itemize}
See \cite{Sogge} for further discussion of cinematic curvature and its properties. Cinematic curvature was first introduced when studying the Bourgain circular maximal function (see e.g.~\cite{Bourgain2}), and it appears that replacing the circles $C(x,r)$ in \eqref{WolffCircMaxmlFnCircles} by families of curves satisfying the cinematic curvature condition is the most natural variable-coefficient generalization of the Wolff circular maximal function. In particular, geodesic circles for a Riemannian metric satisfy the cinematic curvature condition provided that the injectivity radius is larger than the diameter of the circles.

In \cite{Wolff3}, Wolff and Kolasa established the bound
\begin{equation}\label{WolffKolasaBoundEquation}
\norm{M_\Phi^\delta f}_{L^q([1/2,1])}\leq
C_{p,q}\delta^{-\frac{1}{2}(\frac{3}{p}-1)}\norm{f}_{L^p(\RR^2)},\ \ \
p<\frac{8}{3},\ q\leq 2p^\prime.
\end{equation}

In particular, \eqref{WolffKolasaBoundEquation} implies that any compact set containing a curve of the form $\{y\colon\Phi(x,y)=r\}$ for each $0<r<1$ must have Hausdorff dimension at least $11/6$. We shall call such sets \emph{Cinematic BRK sets}.

\subsection{New results}
In this paper, we prove the following theorem:
\begin{thm}\label{mainThm}
Let $\Phi$ satisfy the cinematic curvature conditions \eqref{cinematicCurvatureGradientCondition} and \eqref{cinematicCurvatureCondition}. Then for all $\epsilon>0$ there exists a constant $C_\epsilon$ such that
\begin{equation}
\norm{M_\Phi^\delta f}_{L^3([1/2,1])}\leq C_\epsilon\delta^{-\epsilon}\norm{f}_{L^3(\RR^2)}.
\end{equation}
In particular, every cinematic BRK set must have Hausdorff dimension 2.
\end{thm}

\begin{cor}\label{CircularMaxmlFnBdCor}
Equation \eqref{L3BoundCircles} holds with $p=3$.
\end{cor}

\begin{rem}
 While Corollary \ref{CircularMaxmlFnBdCor} was already known (indeed, it was proved by Wolff), our proof avoids some technical lemmas from Wolff's proof and thus our proof is shorter. If one is interested only in the (original) Wolff circular maximal function, the current shortest proof is obtained through the following steps: the $L^3$--boundedness of the Wolff circular maximal function is established in \cite[\S4]{Schlag}, provided a certain hypothesis is met. This hypothesis is established in this paper, using Lemmas 1.8 and 1.10 from \cite{Wolff4}.
\end{rem}

Theorem \ref{mainThm} improves upon a previous result of the author in \cite{Zahl} in which a similar statement is proved under the additional restriction that the function $\Phi$ be algebraic. We follow a similar proof strategy in this paper as in \cite{Zahl}, but at a key step we use the discrete polynomial ham sandwich theorem rather than the vertical algebraic decomposition.

\subsection{Proof sketch}
Through standard techniques, it suffices to obtain certain weak-type bounds on $\sum \chi_{\Gammabar^\delta}$ for a collection of curves $\{\Gammabar\}$ with $\delta$--separated ``radii.'' The main difficulty arises when many pairs of curves are almost tangent, and indeed a result due to Schlag in \cite{Schlag} shows that we can obtain the desired bounds on $M_\Phi$ if we can control the number of such almost-tangencies. More specifically, if $\WC$ and $\BC$ are collections of curves such that all curves in $\WC$ (resp.~$\BC$) are close to each other in a suitable parameter space, and all curves in $\WC$ are far from curves in $\BC$ (again in a suitable parameter space), then we need to control the number of near-tangencies between curves in $\WC$ and curves in $\BC$. We shall do this with an induction argument.

First, we shall use Jackson's theorem to replace the curves $\{\Gammabar\}$ by algebraic curves that closely approximate them. The degree of the algebraic curves will depend on $\delta$, but the dependence is mild enough to be controllable. We will then identify the curves in $\WC$ with points in $\RR^3$ (if the curves were actually circles, we could use the center and radius of the circle to perform this identification). We then use the discrete polynomial ham sandwich theorem to find a low degree trivariate polynomial $P$ whose zero set partitions $\RR^3$ into open ``cells,'' such that the points are evenly split up amongst the cells. To each curve $\Gammabar\in\BC$ we associate a semi-algebraic set  $Q(\Gammabar)\subset\RR^3$ (of controlled degree), such that if $\Gammabar\in\BC$ is almost tangent to $\tilde\Gammabar\in\WC$, then $Q(\Gammabar)$ must intersect the cell containing (the point associated with) $\tilde\Gammabar$. The bounds on the degree of $P$ and $Q(\Gammabar)$ yield bounds on the number of 
cells that $Q(\Gammabar)$ can intersect. We then apply the induction hypothesis within each cell. Summing over all cells, we obtain the desired bound on the total number of almost-tangencies between curves in $\WC$ and $\BC$.

The key innovation is the use of the discrete polynomial ham sandwich theorem. While the partition of $\RR^3$ described above could be done with the vertical algebraic decomposition instead of the polynomial ham sandwich theorem, the resulting control on the number of cells that $Q(\Gammabar)$ can intersect is so poor that we cannot run the induction argument except in the special case where the defining function $\Phi$ is algebraic (and thus the algebraic curves $\Gammabar$ have degree that does not depend on $\delta$).

The use of the polynomial ham sandwich theorem to solve a non-discrete problem in harmonic analysis might be of interest to readers because to the best of the author's knowledge, the work of Guth \cite{Guth} and Bourgain-Guth \cite{Bourgain} are the only other examples where the polynomial ham sandwich theorem is used to solve a problem of this type.

\subsection{Thanks}
The author would like to thank Javier P\'erez and Terence Tao for pointing out typos in an earlier version of this manuscript. The author was supported in part by the Department of Defense through the National Defense Science \& Engineering Graduate Fellowship (NDSEG) Program.

\section{Proof of Theorem \ref{mainThm}}
\subsection{Initial reductions}
The first step will be to replace the defining function $\Phi$ by an algebraic approximation. This idea was suggested to the author by Larry Guth, and it appears in a similar form in \cite{Bourgain}. Throughout the proof, we shall assume that $\Phi$ satisfies the cinematic curvature conditions at the point $(a,b)=(0,0)$ and that $U_1,U_2$ are small balls centered at 0. By Jackson's theorem (see e.g.~\cite{Bagby}), for each $K>0$, $A>0$, and $\delta>0$, we can find a polynomial $\Psi(x,y)\colon\RR^2\times\RR^2\to\RR$ such that
\begin{equation}\label{degreeOfPsi}
\deg\Psi\leq C_K \delta^{-1/K},
\end{equation}
\begin{equation}\label{howCloseArePhiAndPsi}
\norm{\Psi-\Phi}_{C^2(B(0,100))}<\delta/A,
\end{equation}
where $C_K$ depends on $K,$ $A$, and $\Phi$. If $A$ is chosen sufficiently large depending on the infimum of the quantities in \eqref{cinematicCurvatureGradientCondition} and \eqref{cinematicCurvatureCondition},
then $\Psi$ satisfies \eqref{cinematicCurvatureGradientCondition} and \eqref{cinematicCurvatureCondition}.

Since $\norm{\nabla_y\Phi}$ and $\norm{\nabla_y\Psi}$ are bounded from below for $y\in U_1$ (after possibly shrinking $U_1$), we have that if $A$ is chosen sufficiently large in \eqref{howCloseArePhiAndPsi} then for each $x_0\in U_2$ and $1/2 < r_0 <1$ we have that $\{y\in U_1\colon \Phi(x_0,y)=r_0\}$ and $\{y\in U_1\colon \Psi(x_0,y)=r_0\}$ are contained in $\delta/100$ neighborhoods of each other. Thus if $f$ is supported in $B(0,1)$ then $M_\Phi^\delta f \sim M_\Psi^\delta f$, so it suffices to obtain bounds on $M_\Psi^\delta f.$

\begin{rem}
 If the reader is only interested in the original Wolff circular maximal function, then this step can be omitted, and every instance of $\Psi$ can be replaced by $\Phi(x,y) = \norm{x-y}$. In this case, $\Psi$--circles are arcs of genuine circles. Throughout the proof, we shall refer to this situation as the ``circles'' case.
\end{rem}

Fix $\alpha>0$ sufficiently small depending on the quantities appearing in \eqref{cinematicCurvatureGradientCondition} and \eqref{cinematicCurvatureCondition} and on $\norm{\Phi}_{C^3(B(0,100))}$. For $x\in B(0,\alpha)$ and $r\in [1/2,1]$, we define
\begin{equation}\label{gammaDefn}
\Gamma(x_0,r_0)=\{y\in B(0,\alpha)\colon\Psi(x_0,y)=r_0\}.
\end{equation}
We shall call these sets \emph{$\Psi$--circles}, and if $\Gamma$ is a $\Psi$--circle then $\Gamma^\delta$ will denote its $\delta$--neighborhood. If $\Gamma,\tilde\Gamma,$ etc.~are $\Psi$--circles, then unless otherwise noted, $x_0,r_0$ and $\tilde x_0,\tilde r_0$ will refer to their respective centers and radii. The $\Psi$--circles defined here are strict subsets of the analogous sets $\Gammabar$ defined in the introduction. However, if the function $f$ is supported on a sufficiently small neighborhood of the origin then we can define a maximal function analogous to \eqref{CurveMaximalFunctionDefn} with $\Gamma$ in place of $\Gammabar$, and the two maximal functions will agree. Thus we shall henceforth work with curves $\Gamma$ defined by \eqref{gammaDefn}.

We shall restrict our attention to those $\Psi$--circles $\Gamma$ with $x_0\in B(0,\alpha),\ r_0\in(1-\tau,1)$ for $\tau$ a sufficiently small constant which depends only on the quantities appearing in \eqref{cinematicCurvatureGradientCondition} and \eqref{cinematicCurvatureCondition} and on $\norm{\Phi}_{C^2(B(0,100))}$. By standard compactness arguments, we can recover $L^p([1/2,1])$ bounds on $M_\Psi$ from those on the ``restricted'' version of $M_\Psi$ by considering the supremum over a finite number of scaled versions of the function.

Using standard reductions (see e.g.~\cite{Schlag}, \S4), in order to prove Theorem \ref{mainThm} it suffices to prove the following lemma:
\begin{lem}\label{quantitativeMaximalFunctionBound}
For $\eta>0$ and $\delta$ sufficiently small depending on $\eta$, let $\mathcal A$ be a collection of $\Psi$--circles with $\delta$--separated radii, with each radius lying in $(1-\tau,1)$. Then there exists $\tilde{\mathcal A}\subset\mathcal A$ with $\#\tilde{\mathcal A}\geq\frac{1}{C}\# \mathcal A$ such that for all $\Gamma\in\tilde{\mathcal A}$ and $\delta<\lambda<1$,
\begin{equation}
\Big|B(0,\alpha)\cap\{y\in \Gamma^\delta\colon\ \sum_{\tilde\Gamma\in\mathcal A}\chi_{\tilde\Gamma^\delta}(y) >\delta^{-\eta}\lambda^{-2}\}\Big|\leq \lambda|\Gamma^\delta|.
\end{equation}
\end{lem}
\subsection{Schlag's reduction}
We shall recall a result due to Schlag that shows that Lemma \ref{quantitativeMaximalFunctionBound} is implied by a combinatorial lemma controlling the number of almost-incidences between $\Psi$--circles. In order to state Schlag's result, we will first need several definitions.

\begin{defn}
Let $\Gamma$ and $\tilde\Gamma$ be two $\Psi$ circles. We define
\begin{equation}\label{defOfDeltaX}
\Delta(\Gamma,\tilde\Gamma)=\inf_{\substack{y\in B(0,\alpha)\colon \Psi(x_0,y)=r_0\\\tilde y\in B(0,\alpha)\colon\Psi(\tilde x_0,\tilde y)=\tilde r_0}} |y-\tilde y|+\Big|\frac{\nabla_y\Psi(x_0,y)}{\norm{\nabla_y\Psi(x_0,y)}}-\frac{\nabla_y\Psi(\tilde x_0,\tilde y)}{\norm{\nabla_y\Psi(\tilde x_0,\tilde y)}}\Big|.
\end{equation}
\end{defn}
Informally, if $\Delta(\Gamma,\tilde\Gamma)$ is small then there is a point $y\in B(0,\alpha)$ where $\Gamma$ and $\tilde\Gamma$ pass close to each other and are nearly parallel (i.e.~they are nearly tangent).

Let
\begin{equation}
d(\Gamma,\tilde\Gamma)=|x_0-\tilde x_0|+|r_0-\tilde r_0|.
\end{equation}
$d(\cdot,\cdot)$ is a metric on the space of curves.

\begin{defn}\label{defnOfAtBipartitePair}
Let $\mathcal W$ and $\mathcal B$ be collections of $\Psi$--circles. We say that $\WB$ is a \emph{$\dt$--bipartite pair} if
\begin{align}
&|r_0-\tilde r_0|\geq\delta\ \textrm{for all}\ \Gamma,\tilde\Gamma\in\WC\cup\BC,\\
&d(\Gamma,\tilde\Gamma)\in (t,2t)\ \textrm{if}\ \Gamma\in \WC,\ \tilde\Gamma\in\BC,\\
&d(\Gamma,\tilde\Gamma)\in (0,t)\ \textrm{if}\ \Gamma,\tilde\Gamma\in \WC\ \textrm{or}\ \Gamma,\tilde\Gamma\in \BC.
\end{align}
\end{defn}
\begin{defn}\label{defnOfARectangle}
A \emph{$\dt$--rectangle} $R$ is the $\delta$--neighborhood of an arc of length $\sqrt{\delta/t}$ of a $\Psi$--circle $\Gamma$. We say that a $\Psi$--circle $\Gamma$ is \emph{incident} to $R$ if $R$ is contained in the $C_1\delta$ neighborhood of $\Gamma$. We say that $R$ is of type $(\gtrsim\mu,\gtrsim\nu)$ relative to a $\dt$--bipartite pair $\WB$ if $R$ is incident to at least $C\mu$ curves in $\mathcal W$ and at least $C\nu$ curves in $\mathcal B$ for some absolute constant $C$ to be specified later. We say that two $\dt$--rectangles $R_1,R_2$ are \emph{comparable} if $R_1$ is contained in a $A_0\delta$--neighborhood of $R_2$ and vice versa, where $A_0$ is an absolute constant. Otherwise, we say $R_1$ and $R_2$ are \emph{incomparable}.
\end{defn}
We are now able to state Schlag's result.
\begin{prop}[Schlag]\label{schlagsThm}
Let $\mathcal A$ be a family of $\Psi$--circles with $\delta$--separated radii that satisfy the following requirements:
\begin{enumerate}[label=(\roman{*}), ref=(\roman{*})]
\item\label{schlagThmItemOne}\itemizeEqnVSpacing
\begin{equation}\label{schlagThmControlOfIntersectionSize}
|\Gamma^\delta\cap \tilde \Gamma^\delta\cap B(0,\alpha)|\lesssim\frac{\delta^2}{(d(\Gamma,\tilde\Gamma)+\delta)^{1/2}(\Delta(\Gamma,\tilde\Gamma)+\delta)^{1/2}}.
\end{equation}
\item\label{schlagThmItemTwo} Fix $\epsilon>0.$ Then there exists a constant $C_\epsilon$ so that for any $\dt$--bipartite pair $\WB$, with $t>C\delta$ for an appropriate choice of $C$; $\WC,\BC\subset\mathcal A;\ \#\WC=m;$ and $\#\BC=n,$ the maximum number of pairwise incomparable $\dt$--rectangles of type $(\gtrsim\mu,\gtrsim\nu)$ relative to $\WB$ is at most
\begin{equation}\label{schlagThmNumberIncidencesControl}
C_\epsilon\delta^{-\epsilon}\Big(\Big(\frac{mn}{\mu\nu}\Big)^{3/4}+\frac{m}{\mu}+\frac{n}{\nu}\Big).
\end{equation}
\end{enumerate}
Then Lemma \ref{quantitativeMaximalFunctionBound} holds for the collection $\mathcal A$.
\end{prop}
\begin{rem}
Schlag uses the stronger bound
\begin{equation}\label{strongerSchlagBound}
C_\epsilon(mn)^{\epsilon}\Big(\Big(\frac{mn}{\mu\nu}\Big)^{3/4}+\frac{m}{\mu}+\frac{n}{\nu}\Big)
\end{equation}
in place of \eqref{schlagThmNumberIncidencesControl}. However, an examination of the proof in \cite{Schlag} reveals that the bound \eqref{schlagThmNumberIncidencesControl} suffices. If we restrict our attention to the original Wolff circular maximal function (i.e.~if we are only concerned with the circles case), then we obtain the bound \eqref{strongerSchlagBound}, so Schlag's result can be used as a black box.
\end{rem}
Property \ref{schlagThmItemOne} follows from \cite[Lemma 3.1(i)]{Wolff3}, but if the reader is only interested in the original Wolff circular maximal function, a shorter proof can be found in \cite[\S 3]{Wolff2}. Property \ref{schlagThmItemTwo} follows from the following lemma, which is an analogue of Lemma 1.4 in \cite{Wolff4}:
\begin{lem}\label{thinkOfNameForLem1}
 Let $\Psi\colon \RR^2\times\RR^2\to \RR$ be a (multivariate) polynomial of degree $k$ satisfying the cinematic curvature requirements. Then for every $\epsilon>0$ there exists a constant $C_\epsilon$ such that if $\WB$ is a $\dt$--bipartite pair of $\Psi$--circles with $\#\mathcal W = m,\ \#\mathcal B = n,$ and if $\mathcal R$ is a collection of pairwise incomparable $\dt$--rectangles of type $(\gtrsim\mu,\gtrsim\nu)$ relative to $\WB$, then
\begin{equation}\label{SchlagFood}
\#\mathcal R \leq C_\epsilon k^{C_\epsilon} (mn)^{\epsilon}\Big(\Big(\frac{mn}{\mu\nu}\Big)^{3/4}+\frac{m}{\mu}+\frac{n}{\nu}\Big).
\end{equation}
\end{lem}

To obtain Property \ref{schlagThmItemTwo} from Lemma \ref{thinkOfNameForLem1}, select $K>C_\epsilon/\epsilon$ in \eqref{degreeOfPsi} and note that $(mn)^\epsilon\leq\delta^{2\epsilon}$. In the case of circles we have $k=O(1)$, and \eqref{SchlagFood} becomes \eqref{strongerSchlagBound}.

Thus all that remains is to prove Lemma \ref{thinkOfNameForLem1}. First, we shall recall several properties of curves satisfying the cinematic curvature condition.
\subsection{Properties of $\Psi$--circles}
\begin{defn}
If $\WB$ is a $\dt$--bipartite pair, then we define $\rect_{\mu,\nu}(\WC,\BC)$ to be the maximum cardinality of a collection of pairwise incomparable rectangles of type $(\gtrsim\mu,\gtrsim\nu)$ relative to $\WB$. Define $\rect(\WC,\BC)$ to be $\rect_{1,1}(\WC,\BC)$.
\end{defn}
\begin{defn}
If $\WB$ is a $\dt$--bipartite pair, then we define
\begin{equation*}
\incidences(\WC,\BC)=\#\{(R,\Gamma,\tilde\Gamma)\colon \Gamma\in\WC,\ \tilde\Gamma\in\BC,\ \textrm{$R$ is incident to}\ \Gamma\ \textrm{and}\ \tilde\Gamma\}.
\end{equation*}
\end{defn}

The following ``Canham threshold'' type result is Lemma 34 from \cite{Zahl} (or in the case of circles, Lemma 1.10 from \cite{Wolff4} again has a shorter proof).
\begin{lem}\label{canhamLem}
Let $\WB$ be a $\dt$--bipartite pair. Then
\begin{equation}\label{CanhamThreshhold}
\rect \WB\lesssim nm^{2/3}+m\log n.
\end{equation}
\end{lem}
In brief, Lemma \ref{canhamLem} relies on a variant of the Marstrand three circle lemma \cite{Marstrand}, which is a quantitative formulation of the classical theorem of Appolonius: given three circles that are not all tangent at a common point, there exist at most two circles that are tangent to each of the three given circles. This observation is combined with the K\H{o}vari-S\'os-Turan theorem from \cite{Turan}. Details are in \cite{Zahl}.
\begin{defn}
A collection $\mathcal C$ of $\Psi$--circles is a \emph{cluster} if there exists a $\dt$--rectangle $R$ such that every $\Gamma\in\mathcal C$ is incident to a $\dt$--rectangle comparable to $R$.
\end{defn}
While a cluster can contain many $\Psi$--circles, if we are interested only in counting incidence rectangles then a cluster behaves like a single $\Psi$--circle. This heuristic is made precise though Lemma 41 from \cite{Zahl}:
\begin{lem}\label{lem41}
Let $\mathcal C \subset \mathcal W$ be a cluster and let $\Gamma\in\mathcal B$. Then then any set of pairwise incomparable $\dt$--rectangles, each of which is tangent to some $\Psi$--circle in $\mathcal C$ and to $\Gamma,$ has cardinality $O(1)$.
\end{lem}

The following is Lemma 43 from \cite{Zahl}:
\begin{lem}\label{lemma115}
Let $\WB$ be a $\dt$--bipartite pair. Given a value of $\mu_0$, we can write
\begin{equation}
\WC=\WC_g\sqcup\WC_b,
\end{equation}
where
\begin{enumerate}[label=(\roman{*}), ref=(\roman{*})]
\item\label{lemma115Prop1} $(\WC_g,\BC)$ have no
$\dt$--rectangles of type
$(\gtrsim\mu_0,\gtrsim1)$.%
\item\label{lemma115Prop2} $\WC_b$ is the union of
$\lesssim\frac{\#\WC}{\mu_0}(\log m)(\log n)$ clusters.
\end{enumerate}
\end{lem}

If every $\Psi$--circle from $\WC$ and $\BC$ are incident to some common rectangle $R$ then $\incidences(\WC,\BC)=(\#\WC)(\#\BC)$. However, if neither $\WC$ nor $\BC$ contain large clusters then this cannot occur.  Lemma 36 from \cite{Zahl} is a quantitative version of this observation:
\begin{lem}\label{lemma111}
Let $\WB$ be a $\dt$--bipartite pair that has no
$(\gtrsim1, \gtrsim\nu_0)$ or $(\gtrsim\mu_0,\gtrsim1)$ rectangles
$R\in B(b,\alpha)$. Then
\begin{equation}
\incidences(\mathcal W,\mathcal
B)\lesssim\mu_0^{1/3}nm^{2/3}\log \nu_0+\nu_0m\log\mu_0.
\end{equation}
\end{lem}
\subsection{Algebraic considerations}
We shall identify the $\Psi$--circle $\Gamma$ with the point $(x_0,r_0)\in\RR^3$ (actually in $B(0,\alpha)\times(1-\tau,1)\subset\RR^3$). Thus if $\WC$ is a collection of $\Psi$--circles, we shall abuse notation and simultaneously consider $\WC$ as a subset of $\RR^3$.
\begin{lem}\label{constructionOfSandwich}
Let $\Psi\colon\RR^2\times\RR^2\to\RR$ be a (multivariate) polynomial of degree $k$ that satisfies the cinematic curvature conditions. For each $\Psi$--circle $\Gamma$, there exists a set $Q(\Gamma)\subset\RR^3$ with the following properties:
\begin{enumerate}[label=(\roman{*}), ref=(\roman{*})]
\item\label{sandwichProperty1} $\bdry Q(\Gamma)$ is contained in an algebraic set $S_\Gamma$ of dimension $2$ and complexity $O(k^C)$ (see Appendix \ref{realAlgGeoAppendix} for relevant definitions).%
\item\label{sandwichProperty2} Let $\tilde\Gamma$ be a $\Psi$--circle with $d(\Gamma,\tilde\Gamma)>A\delta$ for $A$ a sufficiently large constant. If $\tilde\Gamma\in Q(\Gamma)$ then $\Delta(\Gamma,\tilde\Gamma)\leq100\delta$. Conversely, if $\Delta(\Gamma,\tilde\Gamma)<\delta$ then $\tilde\Gamma\in Q(\Gamma)$.
\end{enumerate}
\end{lem}
\begin{rem}
Informally, $Q(\Gamma)$ can be understood as follows. If $\gamma_1=C(x_1,r_1),$ $\gamma_2=C(x_2,r_2)$ are two circles, then $\gamma_1$ and $\gamma_2$ are tangent if and only if $(x_2,r_2)$ lies on the right-angled light-cone $Z_{\gamma_1}=\{(y,t)\colon |r-t| = \norm{x-y}\}$, and $\gamma_1$ and $\gamma_2$ are almost tangent if $(x_2,r_2)$ lies in the $\delta$--neighborhood of $Z_{\gamma_1}$. $Q(\Gamma)$ is the analogue of the $\delta$--neighborhood of the light cone $Z_{\gamma_1}$ for general curves $\Gamma$.
\end{rem}
\begin{proof}
Define
\begin{equation}
V_{\Gamma}=V_{1,\Gamma}\cap V_{2,\Gamma}\cap V_{3,\Gamma}\cap V_{4,\Gamma},
\end{equation}
where
\begin{align*}
V_{1,\Gamma}=&\ \{(\tilde x_0,\tilde r_0,y,\tilde y)\colon \norm{\tilde x_0}^2<\alpha^2,\ 0<1-\tilde r_0<\tau,\ \norm{y}^2<\alpha^2,\  \norm{\tilde y}^2<\alpha^2\},\\
V_{2,\Gamma}=&\ \{(\tilde x_0,\tilde r_0,y,\tilde y)\colon\Psi(x_0,y)=r_0,\ \Psi(\tilde x_0,\tilde y)=\tilde r_0\},\\
V_{3,\Gamma}=&\ \{(\tilde x_0,\tilde r_0,y,\tilde y)\colon \norm{y-\tilde y}^2 < \delta^2\},\\
V_{4,\Gamma}=&\ \{(\tilde x_0,\tilde r_0,y,\tilde y)\colon \norm{\nabla_y \Psi(x_0,y) \wedge \nabla_y \Psi(\tilde x_0,\tilde y)}^2 \\
&\phantom{\{(\tilde x_0,\tilde r_0,y,\tilde y)\colon \norm{\nabla_y \Psi(x_0,y)}} <4\delta^2 \norm{\nabla_y \Psi(x_0,y)}^2\norm{\nabla_y \Psi(\tilde x_0,\tilde y)}^2\}.
\end{align*}
Each $V_{j,\Gamma},\ j=1,2,3,4$ is a semi-algebraic set of complexity $O(k^C)$ (see Appendix \ref{realAlgGeoAppendix} for relevant definitions), and thus so is $V_{\Gamma}.$ Let
\begin{equation}
Q(\Gamma)=(\pi_{(\tilde x_0,\tilde r_0)}V_{\Gamma})\cap\{\tilde x_0\colon \norm{x_0-\tilde x_0}^2>A^2\delta^2\},
\end{equation}
where $\pi_{(\tilde x_0,\tilde r_0)}\colon (\tilde x_0,\tilde r_0,y,\tilde y)\mapsto (\tilde x_0,\tilde r_0)$ is the projection map.

An examination of the definition of $\Delta(\Gamma,\tilde\Gamma)$ verifies that $Q(\Gamma)$ satisfies Property \ref{sandwichProperty2}, so all that remains is to verify Property \ref{sandwichProperty1}. Since $V_{\Gamma}$ is a semi-algebraic set of complexity $O(k^C)$, by the Tarski-Seidenberg theorem (see Proposition \ref{effectiveTarskiProp} in Appendix \ref{realAlgGeoAppendix}), so is $Q(\Gamma)$. Thus by Proposition  \ref{propertiesOfTheBoundaryProp} in Appendix \ref{realAlgGeoAppendix}, either $Q(\Gamma)$ is empty or $\bdry(Q(\Gamma))$ has dimension at most 2 and complexity $O(k^C)$, so by Proposition \ref{propertiesOfZariskiClosureProp}, its Zariski closure, $Z_\Gamma=\zar(\bdry(Q(\Gamma))),$ is an algebraic set of dimension at most 2 and degree $O(k^C)$. If $\dim(Z_\Gamma)=2$ then let $S_\Gamma=Z_\Gamma$. If not, we can find an algebraic set of dimension 2 containing $Z_\Gamma$ whose degree is controlled by a polynomial function of the degree of $Z_\Gamma$ and we shall let this set be $S_\Gamma$.
\end{proof}

\begin{defn}
 Let $\WC\subset\RR^N$ be a finite collection of points. We say that $\WC$ is \emph{hypersurface generic} if for every polynomial $P\in\RR[x_1,\ldots,x_N]$ of degree $D$ we have
\begin{equation*}
 \#(\{P=0\}\cap\WC)\leq\binom{D}{N}-1.
 \end{equation*}
\end{defn}
\begin{lem}
 Let $\WC\subset\RR^3$ be finite. Then after an infinitesimal perturbation, $\WC$ is hypersurface generic
\end{lem}
\begin{proof}
Identify the space of all sets $H\subset\RR^3$ of cardinality $\ell$ with $(\RR^N)^\ell$. Let $\#\WC=m$. Then the subset of $(\RR^N)^{m}$ corresponding to sets of cardinality $m$ that are not hypersurface generic is Zariski closed---it is a finite union of determinantal varieties.
\end{proof}

We shall now recall a corollary of the discrete polynomial ham sandwich theorem. A proof of this theorem (and of the corollary) can be found in \cite[Theorem 4.1]{Guth2}.

\begin{prop}[Polynomial cell decomposition]\label{discPolyHamSandThm}
Let $\mathcal W\subset\RR^N$ be a collection of points. Then for each $D>0$, there exists a polynomial $P\in\RR[x_1,\ldots,x_N]$ of degree at most $D$ such that $\RR^N\backslash\{P=0\}$ is a union of $\lesssim D^N$ open connected sets (henceforth ``cells''), and for each cell $\Omega,$ we have
\begin{equation}
\#(\mathcal W \cap\Omega)\lesssim \#\mathcal W/D^N.
\end{equation}
\end{prop}

\subsection{Proof of Lemma \ref{thinkOfNameForLem1}}
In order to prove Lemma \ref{thinkOfNameForLem1}, it suffices to consider the case where $\mu=\nu=1$ and establish the following bound:

\begin{lem}\label{muNuEqualsOneCase}
Let $\WB$ be as in Lemma \ref{thinkOfNameForLem1}. Then for all $\epsilon>0$, there exists a constant $C_\epsilon$ such that
\begin{equation}\label{biPartitePairRectControlSpecialCase}
\rect(\WC,\BC)\leq C_\epsilon k^{C_\epsilon}(mn)^\epsilon ((mn)^{3/4}+m+n).
\end{equation}

\end{lem}

To obtain \eqref{SchlagFood} from \eqref{biPartitePairRectControlSpecialCase} we apply a random sampling argument. The details can be found in \cite[p1253]{Wolff4}, so we shall not reproduce them here.
\begin{proof}[Proof of Lemma \ref{muNuEqualsOneCase}]
We shall proceed by induction on the quantity $(\#\WC)(\#\BC)$.
To handle the base case, we may assume
\begin{equation}\label{mIsBig}
mn> C_\epsilon k^{C_\epsilon},
\end{equation}
since otherwise we can use the trivial bound $\rect(\WC,\BC) \lesssim mn$. Now suppose Lemma \ref{muNuEqualsOneCase} has been established for all $\dt$--bipartite pairs $(\WC^\prime,\BC^\prime)$ with $(\#\WC^\prime)(\#\BC^\prime)<mn$.

We may assume
\begin{equation}\label{mAndNNotTooDifferent}
 Am^{1/3+\epsilon}<n<m,
\end{equation}
for a large constant $A$ (depending on $\epsilon$) to be determined later, since if the first inequality fails then the result follows from \eqref{CanhamThreshhold} (and after selecting a sufficiently large value of $C_\epsilon$, depending on $A$), while if the second inequality fails we can reverse the roles of $\WC$ and $\BC$.

Let $\mu_0 = (mn)^{1/4},$ and use Lemma \ref{lemma115} to write $\WC=\WC_g\sqcup\WC_b$ and similarly $\BC=\BC_g\sqcup\BC_b$. Using Lemma \ref{lem41}, we have
\begin{align}
\rect(\WC_g,\BC)&\leq\frac{1}{100}(mn)^{3/4+\epsilon},\\
\rect(\WC,\BC_g)&\leq\frac{1}{100}(mn)^{3/4+\epsilon}.
\end{align}
See \cite[p1251-2]{Wolff4} for details. Thus in order to prove Lemma \ref{muNuEqualsOneCase}, it suffices to establish the following bound:
\begin{equation}\label{controlOfWgBg}
\rect(\WC_g,\BC_g)<\frac{1}{2}C_\epsilon K^{C_\epsilon}(mn)^{3/4+\epsilon}+C_\epsilon K^{C_\epsilon}(mn)^{\epsilon}(m+n).
\end{equation}

Use Proposition \ref{discPolyHamSandThm} to select a polynomial $P\in\RR[x_1,x_2,r]$ of degree at most $D$ ($D$ shall be chosen later, but it should be thought of as $\delta^{\epsilon}$) so that the set $\RR^3\backslash\{P=0\}$ is a union of $\lesssim D^3$ cells, each of which contains $\lesssim \#\WC_g /D^3$ $\Psi$--circles $\Gamma\in\WC_g$.
\begin{lem}\label{cellMeetsBoundaryOrIsContained}
Let $\Omega$ be a cell from the above decomposition. If $\Gamma\in\BC_g,\ \tilde\Gamma\in\Omega$, and $\Delta(\Gamma,\tilde\Gamma)\leq\delta,$ then at least one of the following must hold.
\begin{enumerate}[label=(\roman{*}), ref=(\roman{*})]
 \item\label{cellMeetsBoundaryOrIsContainedItemOne} $\bdry Q(\Gamma)\cap\Omega\neq\emptyset$.
 \item\label{cellMeetsBoundaryOrIsContainedItemTwo} $\Omega\subset Q(\Gamma)$.
\end{enumerate}
\end{lem}

Indeed, since $\Delta(\Gamma,\tilde\Gamma)\leq\delta,$ by Property \ref{sandwichProperty2} of $Q(\Gamma)$ from Lemma \ref{constructionOfSandwich}, $\tilde\Gamma\in Q(\Gamma)$ and thus $\Omega\cap Q(\Gamma)\neq\emptyset$. Since $\Omega$ is an open connected set, it must either be contained in $Q(\Gamma)$ or it must meet the (topological) boundary of $Q(\Gamma)$.

Now, for each cell $\Omega$, let
\begin{equation*}
\BC_g = \BC^\Omega_1\sqcup\BC^\Omega_2\sqcup\BC^\Omega_3,
\end{equation*}
 where $\BC^\Omega_1$ (resp.~ $\BC^\Omega_2$) contains those $\Gamma\in\BC_g$ for which Item \ref{cellMeetsBoundaryOrIsContainedItemOne} (resp.~Item \ref{cellMeetsBoundaryOrIsContainedItemTwo}) occurs, and $\Gamma\in\BC^\Omega_3$ if $\Delta(\Gamma,\tilde\Gamma)>\delta$ for all $\tilde\Gamma\in\Omega$.

We shall first consider incidences involving $\BC^\Omega_2$.
\begin{lem}\label{BC2SmallOrWcgSmall}
Suppose $D$ satisfies
\begin{equation}\label{sizeOfD}
D < n^{\epsilon/6}.
\end{equation}
Then if $m$ and $n$ are sufficiently large, at least one of the following must hold:
\begin{align}
&\#\Big(\bigcup_{\Omega} \BC_2^\Omega\Big) < n/1000,\label{BC2NotTooBig}\\
&\#\WC_g < m/1000\label{BCgNotTooBig}.
\end{align}
\end{lem}
\begin{proof}
Suppose \eqref{BCgNotTooBig} fails. By \eqref{mAndNNotTooDifferent}, \eqref{sizeOfD}, and the fact that $\WC$ is hyperplane generic, we have that for each cell $\Omega,$
\begin{equation*}
\begin{split}
\#(\WC_g\cap\Omega)&\gtrsim \#(\WC_g)D^{-3}\\
&\gtrsim mD^{-3}.
\end{split}
\end{equation*}
Thus each $\Gamma\in\bigcup_{\Omega} \BC_2^\Omega$ is incident to $\gtrsim mD^{-3}$ $\Psi$--circles from $\WC_g$, so
\begin{equation}\label{lowerBdOnIncidences}
\incidences(\WC_g,\BC_g) \gtrsim  m D^{-3} \#\Big(\bigcup_{\Omega} \BC_2^\Omega\Big).
\end{equation}
On the other hand, by Lemma \ref{lemma111} (with $\mu_0=\nu_0=(mn)^{1/4}$),
\begin{equation}\label{upperBdOnIncidences}
\incidences(\WC_g,\BC_g) \lesssim m^{5/4}n^{1/4}\log m + m^{3/4}n^{13/12}\log n.
\end{equation}
Combining \eqref{lowerBdOnIncidences}, \eqref{upperBdOnIncidences}, and \eqref{mAndNNotTooDifferent}, we obtain
\begin{equation}
 \#\Big(\bigcup_{\Omega} \BC_2^\Omega\Big) \lesssim D^3 n^{1-\epsilon}\log n.
\end{equation}
This and \eqref{mIsBig}, \eqref{sizeOfD} gives us \eqref{BC2NotTooBig}.
\end{proof}

If either \eqref{BC2NotTooBig} or \eqref{BCgNotTooBig} holds, then we can apply the induction hypothesis to the pair $(\WC_g, \bigcup_{\Omega} \BC_2^\Omega)$ and conclude that
\begin{equation}\label{controlWCgBC2}
\begin{split}
\rect(\WC_g, \bigcup_{\Omega} \BC_2^\Omega) &\leq \frac{1}{100}C_{\epsilon}k^{C_{\epsilon}}(mn)^{\epsilon}((mn)^{3/4}+m+n)\\
&\leq \frac{1}{10}C_{\epsilon}k^{C_{\epsilon}}(mn)^{3/4+\epsilon},
\end{split}
\end{equation}
where on the second line we used \eqref{mAndNNotTooDifferent}.

\begin{rem}
Lemma \ref{BC2SmallOrWcgSmall} is an analogue of Equation (5.23) from \cite{Zahl}. In essence, both state that if $\#\WC_g$ were too big then that would force an illegally large number of incidences to occur. However, the current formulation is much simpler. In  \cite{Zahl}, the analogue of $Q(\Gamma)$ was defined differently and thus we needed statements of the form ``if two curves $\Gamma_1,\Gamma_2$ are almost tangent then after a slight perturbation they are exactly tangent.'' Making statements such as this rigorous introduced many technical difficulties that have been avoided in the present paper.
\end{rem}

We shall now control incidences involving $\BC^\Omega_2$. Let
\begin{equation*}
n_\Omega=\#\{\Gamma\in\BC_g\colon \bdry(Q(\Gamma))\cap\Omega\neq\emptyset\}.\\
\end{equation*}
Since $\bdry(Q(\Gamma))\subset S_{\Gamma}$, we have
\begin{equation*}
n_\Omega\leq \#\{\Gamma\in\BC_g\colon S_\Gamma \cap\Omega\neq\emptyset\}.\\
\end{equation*}
By a Thom-Milnor type theorem (see e.g.~\cite[Theorem 1.1]{Barone}), we have that for each $\Gamma\in\BC_g$, $S_\Gamma\backslash \{P=0\}$ contains $O(k^CD^2)$ connected components. Since the number of cells that intersect $\bdry(Q(\Gamma))$ is bounded by the number of connected components of $S_\Gamma\backslash \{P=0\}$, we have
\begin{equation}\label{ThomMilnorBound}
 \sum_\Omega n_\Omega \leq C_1 D^2 k^{C} n.
\end{equation}
Let $m_\Omega=\#(\WC_g\cap\Omega)$. Applying the induction hypothesis,
\begin{equation}\label{rectWcBc1}
\begin{split}
 \sum_{\Omega}&\rect(\WC_g\cap\Omega,\BC^{\Omega}_1)\\
&\leq k^{C_\epsilon}C_\epsilon\bigg[ \sum m_\Omega^{3/4+\epsilon} n_{\Omega}^{3/4+\epsilon} +(mn)^{\epsilon}\sum m_\Omega+ (mn)^{\epsilon}\sum n_{\Omega}\bigg] \\
&\leq C_\epsilon k^{C_\epsilon}\bigg[\Big(\sum m_\Omega^{\frac{4}{1-4\epsilon}(3/4+\epsilon)}\Big)^{\frac{1-4\epsilon}{4}}\Big(\sum n_{\Omega}\Big)^{3/4+\epsilon}  \\
&\phantom{\leq C_\epsilon k^{C_\epsilon}\bigg[} +(mn)^{\epsilon}\sum m_\Omega+ \sum (mn)^{\epsilon}n_{\Omega}\bigg]\\
&\leq C_\epsilon k^{C_\epsilon}\bigg[\Big(D^3 m^{\frac{3+4\epsilon}{1-4\epsilon}} D^{-\frac{9+12\epsilon}{1-4\epsilon}}\Big)^{\frac{1-4\epsilon}{4}}(C_1D^2k^C n)^{3/4+\epsilon}  \\
& \phantom{\leq C_\epsilon k^{C_\epsilon}\bigg[} + (mn)^{\epsilon}m + (mn)^{\epsilon}C_1D^2k^Cn\bigg]\\
&= C_\epsilon k^{C_\epsilon}(mn)^{\epsilon}\bigg[\frac{C_1(mn)^{3/4}k^C}{D^{2\epsilon}}+m+C_1D^2k^Cn\bigg].
\end{split}
\end{equation}

Finally, since the points of $\WC$ are hypersurface generic, we have that
\begin{equation*}
\#(\WC_g\cap\{P=0\})\lesssim D^3,
\end{equation*}
and thus
\begin{equation}\label{rectWcP0}
 \rect(\WC_g\cap\{P=0\},\BC_g)\leq C_2D^3 n.
\end{equation}

We have
\begin{equation}
 \rect(\WC_g,\BC_g) = \sum_{\Omega}\rect(\WC_g\cap\Omega,\BC^{\Omega}_1)+ \sum_{\Omega}\rect(\WC_g\cap\Omega,\BC^{\Omega}_2)+ \rect(\WC_g\cap\{P=0\},\BC_g).
\end{equation}

Combining \eqref{controlWCgBC2}, \eqref{rectWcBc1}, and \eqref{rectWcP0}, we conclude that there there exists an absolute constant $C_0$ such that
\begin{equation}\label{controlOfIncidences}
 \rect(\WC_g,\BC_g) \leq C_\epsilon k^{C_\epsilon} (mn)^\epsilon\Big(\frac{C_1 (mn)^{3/4}k^{C_0}}{D^{2\epsilon}} + C_2 D^3k^{C_0} n + m\Big).
\end{equation}
Now, select $D>1$ satisfying \eqref{sizeOfD} and also
\begin{align}
 \frac{C_1k^{C_0}}{D^{2\epsilon}}&<\frac{1}{100},\label{firstDcondition}\\
 C_2D^3k^{C_0}n &< \frac{(mn)^{3/4}}{100}\label{secondDcondition}.
\end{align}
The existence of such a $D$ is guaranteed by \eqref{mIsBig} and \eqref{mAndNNotTooDifferent} provided we select the constants $C_\epsilon$ (from \eqref{mIsBig}) and $A$ (from \eqref{mAndNNotTooDifferent}) to be sufficiently large (depending on the constant $C_0$ from \eqref{controlOfIncidences} and the $\epsilon$ that appears in the statement of Lemma \ref{muNuEqualsOneCase}). With such a choice of $D$, \eqref{controlOfWgBg} is satisfied. This completes the proof of Lemma \ref{muNuEqualsOneCase} and hence also Theorem \ref{mainThm}.
\end{proof}
\begin{remark}
The use of a ``low degree'' partitioning polynomial to prove incidence theorems was first introduced by Solymosi and Tao in \cite{Solymosi}. What we do here is very similar, except instead of using a bounded degree variety and the general heuristic that operations such as projection, etc.~send bounded degree varieties to bounded degree varieties, we use a variety of ``sub-polynomial'' degree, and we rely on the heuristic that projections, etc.~send sub-polynomial degree varieties to sub-polynomial degree varieties.
\end{remark}

\appendix
\section{Real algebraic geometry}\label{realAlgGeoAppendix}
We recall several facts from real algebraic geometry. See e.g.~\cite{Basu,Bochnak} for additional information.

\begin{defn}
A set $S\subset\RR^n$ is \emph{semi-algebraic} if it can be expressed in the form
\begin{equation}\label{semiAlgebraicSetDefnEqn}
S=\bigcup_{i=1}^n \{x\colon f_{i,1}(x)=0,\ldots
f_{i,\ell_i}(x)=0,g_{i,1}(x)>0,\ldots,g_{i,m_i}(x)>0\}
\end{equation}
for $\{f_{i,j}\}$ and $\{g_{i,j}\}$ polynomials.
\end{defn}
\begin{defn}
For $S$ a semi-algebraic set, the \emph{complexity} of $S$ is
\begin{equation}
\inf\Big(\sum \deg f_{i,j}+\sum\deg g_{i,j}\Big),
\end{equation}
where the infimum is taken over all representations of $S$ of the form \eqref{semiAlgebraicSetDefnEqn}.
\end{defn}

\begin{defn}
For $S$ a semi-algebraic set, we define the \emph{boundary} $\bdry(S)=\overline S\backslash S,$ where $\overline S$ is the closure of $S$ in the Euclidean topology.
\end{defn}
\begin{prop}\label{propertiesOfTheBoundaryProp}
$\bdry(S)$ is semi-algebraic, $\dim(\bdry(S))\leq\dim(S)-1$, and the complexity of $\bdry(S)$ is controlled by a polynomial function of the complexity of $S$.
\end{prop}
\begin{defn}
For $S$ a semi-algebraic set, we define its \emph{Zariski closure} $\zar(S)$ to be the closure of $S$ in the (real) Zariski topology
\end{defn}
\begin{prop}\label{propertiesOfZariskiClosureProp}$\phantom{1}$
\begin{enumerate}[label=(\roman{*}), ref=(\roman{*})]%
\item\label{propertiesOfZariskiClosurePropItem1} $\zar(S)$ is an algebraic set.
\item\label{propertiesOfZariskiClosurePropItem2} $\dim(\zar(S))=\dim(S)$.
\item\label{propertiesOfZariskiClosurePropItem3} $\deg(\zar(S))$ is bounded by a polynomial function of the complexity of $S$.
\end{enumerate}
\end{prop}
\begin{proof}
Statements \ref{propertiesOfZariskiClosurePropItem1} and \ref{propertiesOfZariskiClosurePropItem2} are standard. Statement \ref{propertiesOfZariskiClosurePropItem3} follows from the standard properties of the cylindrical algebraic decomposition (see e.g.~\cite{Basu,Bochnak}).
\end{proof}
\begin{prop}[Effective Tarski-Seidenberg Theorem \cite{Collins}]\label{effectiveTarskiProp}
Let $S\subset\RR^d$ be a semi-algebraic set of complexity $k$ and let $\pi\colon \RR^d\to\RR^{d-1}$ be the projection onto the first $d-1$ coordinates. Then $\pi(S)$ is a semi-algebraic set of complexity at most $k^C$ for some constant $C$ that depends only on $d$.
\end{prop}
\bibliographystyle{amsplain}

\end{document}